\documentclass[sigconf]{acmart}

\usepackage{mathtools}
\usepackage{amsfonts}
\usepackage{physics}

\usepackage[T1]{fontenc}
\usepackage[utf8]{inputenc}
\usepackage[english]{babel}
\usepackage{microtype}
\usepackage[autostyle]{csquotes}

\usepackage{enumitem}
\usepackage{hyperref}
\usepackage{tabularx}

\usepackage[english, algosection, algoruled, noline]{algorithm2e}

\usepackage{cleveref}
\usepackage{xcolor}

\newcommand{\zerodisplayskips}{%
  \setlength{\abovedisplayskip}{0.5pt}%
  \setlength{\belowdisplayskip}{0.5pt}%
  \setlength{\abovedisplayshortskip}{0.5pt}%
  \setlength{\belowdisplayshortskip}{0.5pt}}
\appto{\normalsize}{\zerodisplayskips}
\appto{\small}{\zerodisplayskips}
\appto{\footnotesize}{\zerodisplayskips}

\newcommand{\N}{\mathbb N}

\renewcommand{\C}{\mathbb C}

\newcommand{\R}{\mathbb R}

\newcommand{\K}{\mathbb K}

\renewcommand{\epsilon}{\varepsilon}
\newcommand{\p}{\mathfrak p}

\newcommand{\m}{\mathfrak m}

\newcommand{\cI}{\mathcal{I}}

\newcommand{\cL}{\mathcal{L}}

\newcommand{\cV}{\mathcal{V}}

\newcommand*{\tigen}[2]{\langle {{#1}} \rangle_{{#2}}}

\DeclareMathOperator{\sign}{sign}

\DeclareMathOperator{\closure}{cl}

\DeclareMathOperator{\jac}{Jac}

\def \mom {\mathrm{MoM}}

\def \eval {\mathbf{e}}

\setlist[enumerate,1]{label={(\roman*)},ref={\thetheorem (\roman*)}}

\def\ann{\mathrm{Ann}}
\def\CC{\mathbb{C}}
\def\NN{\mathbb{N}}
\def\RR{\mathbb{R}}

\def\RRg{\RR[\vb x]}
\def\CRg{\CC[\vb x]}

\newcommand{\val}[2]{\langle {{#1},{#2}} \rangle}
\newcommand{\vspan}[1]{\langle {#1} \rangle}

\def\proj{\textup{proj}}
\newcommand{\assign}{:=}
\newcommand{\cdummy}{\cdot}

\newcommand{\strong}[1]{\textbf{#1}}

\newcommand{\conj}[1]{\overline{#1}}
\newcommand{\cl}[1]{\closure \left( {#1} \right)}

\newtheorem{theorem}{Theorem}[section]
\newtheorem{remark}[theorem]{Remark}

\copyrightyear{2021}
\acmYear{2021}
\setcopyright{acmlicensed}\acmConference[ISSAC '21]{Proceedings of the 2021 International Symposium on Symbolic and Algebraic Computation}{July      18--23, 2021}{Virtual Event, Russian Federation}
\acmBooktitle{Proceedings of the 2021 International Symposium on Symbolic and Algebraic Computation (ISSAC '21), July      18--23, 2021, Virtual Event, Russian Federation}
\acmPrice{15.00}
\acmDOI{10.1145/3452143.3465541}
\acmISBN{978-1-4503-8382-0/21/07}


\begin{document}
\fancyhead{}
\title{Computing Real Radicals by Moment Optimization}

\author{Lorenzo Baldi \& Bernard Mourrain}
\affiliation{
Inria M\'editerran\'ee, Université C\^ote d'Azur,\\
Sophia
Antipolis,
France \\
}

\begin{abstract}
We present a new algorithm for computing the real radical of an
ideal $I$ and, more generally, the $S$-radical of $I$, which is
based on convex moment optimization. A truncated positive generic linear
functional $\sigma$ vanishing on the generators of $I$ is computed
solving a Moment Optimization Problem (MOP). We show that, for a
large enough degree of truncation, the annihilator of $\sigma$
generates the real radical of $I$. We give an effective, general
stopping criterion on the degree to detect when the prime ideals lying
over the annihilator are real and compute the real radical
as the intersection of real prime ideals lying over $I$.

The method involves several ingredients, that exploit the properties
of generic positive moment sequences. A new efficient algorithm
is proposed to compute a graded basis of the annihilator of a
truncated positive linear functional. We propose a new algorithm to
check that an irreducible decomposition of an algebraic variety is real,
using a generic real projection to reduce to the hypersurface
case. There we apply the Sign Changing Criterion, effectively
performed with an exact MOP.
Finally we illustrate our approach in some examples.
\end{abstract}


\maketitle

\section{Introduction}

In many ``real world'' problems which can be modeled by polynomial
constraints, the solutions with real coordinates are generally
analyzed with particular attention. Efficient algebraic methods have
been developed over the years to solve such systems of polynomial constraints,
including Grobner basis, border basis, resultants, triangular sets, homotopy
continuation. But  all these methods involve implicitly the complex roots of the
polynomial systems and their complexity depends on the degree (and
multiplicity) of the underlying complex algebraic varieties.

Finding equations vanishing on the real solutions without computing
all the complex roots is a challenging question. This means computing
the vanishing ideal of the real solutions of an ideal $I$, that is,
its real radical $\sqrt[\R]{I}$.

Several approaches have been proposed to compute the real
radical.
Some of these methods
are reducing to univariate problems
\cite{BeckerComputationRealRadicals1993,
  NeuhausComputationrealradicals1998,
  BeckerRealNullstellensatz1999,
  Spangzerodimensionalapproachcompute2008},
or exploiting quantifier elimination techniques
\cite{GalligoComplexityFindingIrreducible1995},
or using infinitesimals
\cite{Roycomplexificationdegreesemialgebraic2002}
or triangular sets and regular chains
\cite{XiaAlgorithmIsolatingReal2002,
  ChenTriangulardecompositionsemialgebraic2013}.

Sums-of-Squares convex optimisation and moment matrices are used in
\cite{lasserre_semidefinite_2008,lasserre_moment_2013}
to compute real radicals, when the set of real solutions is finite.
Some properties of ideals associated to semidefinite programming
relaxations are analysed in \cite{SekiguchiRealidealduality2013},
involving the simple point criterion.
In \cite{Macertificatesemidefiniterelaxations2016b}
a stopping criterion is presented to verify that a Pommaret basis has been
computed from the kernels of moment
matrices involved in Sum of Squares relaxation.
In \cite{BrakeValidatingCompletenessReal2016}, a test based on
sum-of-square decomposition is proposed to verify that polynomials
vanishing on a subset of the semi-algebraic set are in the real radical.

In \cite{SafeyElDinComputingrealradicals2021}, an algorithm
based on rational representations of equidimensional components of
algebraic varieties and singular locus recursion is presented and its
complexity is analysed.

We present a new algorithm for computing the real radical of an
ideal $I$ and, more generally, the $S$-radical of $I$,
which is based on convex moment optimization.
An interesting feature of the approach is that it does not involve the
complex solutions, which are not on a real component
of the algebraic variety $\cV(I)$. \Cref{sec:radicals}
recalls the relationship between vanishing ideals and radicals for
real and complex algebraic varieties.

Generators of the real radical of $I$ are computed from a truncated
generic positive linear functional $\sigma$ vanishing on the
generators of $I$. This truncated linear functional
is computed by solving a Moment Optimization Problem (MOP), as
summarized in \Cref{sec:mop}.

We show that, for a large enough degree of truncation, the annihilator
of $\sigma$ generates the real radical of $I$, suggesting an algorithm
which will compute the annihilator of a generic positive linear
functional for increasing degrees.
Our approach differs from the works in
\cite{lasserre_semidefinite_2008,lasserre_moment_2013}, which apply
for zero-dimensional real ideals using the flat extension property
(see
e.g. \cite{CurtoFlatExtensionsPositive1998,Laurentgeneralizedflatextension2009})
as a stopping criterion: if the flat
extension property holds then the annihilator of $\sigma$ generates
the real radical of $I$, and this criterion is satisfied for a degree big enough.
But the question remained open for positive-dimensional real varieties
(see e.g. \cite[\S~4.3]{LaurentApproachMomentsPolynomial2012}). In
this work, we handle more specifically the positive-dimensional
case. This case has been analysed in
\cite{Macertificatesemidefiniterelaxations2016b}, where a stopping
criterion is proposed to detect when a Pommaret basis has been computed. This test
is generically satisfied for a large enough degree of truncation, but
it does not certify that the basis generates the real radical.

In this work, we give a new effective
stopping criterion to detect when the prime ideals associated to the annihilator
are real and compute equations for the minimal real prime ideals lying over $I$. This criterion is always satisfied for a large enough degree of truncation, and it certifies that the annihilator generates the real radical if the generated ideal has no embedded components.

The method involves several ingredients, that exploit the properties
of generic non-negative moment sequences.

A new efficient algorithm is proposed in \Cref{sec:orthogonal}
to compute a graded basis of the annihilator of a truncated
non-negative linear functional. A new algorithm is presented in \Cref{sec:isreal}
to check that an irreducible decomposition of an algebraic variety is real,
using a generic real projection to reduce to the hypersurface
case and the Sign Changing Criterion, effectively
performed with an exact MOP.

The complete algorithm for computing the real radical of an ideal $I$
as the intersection of real prime ideals is presented in \Cref{sec:computing}.

In \Cref{sec:example}, we illustrate the algorithm by some effective numerical
computation on examples, where the real radical differs
significantly from the ideal $I$.

\section{Varieties and radicals}\label{sec:radicals}

Let $f_{1}, \ldots, f_{s}\in \CC[x_{1}, \ldots, x_{n}]=\CRg$ and let $I=(\vb f) \subset \CRg$
be the ideal generated by $\vb f=\{f_{1}, \ldots, f_{s}\}$. The algebraic
variety defined by $\vb f$ is denoted $V=\cV_{\CC}(I)=\{\xi \in \CC^{n}\mid
f_{i}(\xi)=0, i=1,\ldots,s\}$.
It decomposes into an union of irreducible components
$V=\cup_{i=1}^{l} V_{i}$ where $V_{i}=\cV_{\CC}(\p_{i})$ with $\p_{i}$ a
prime ideal of $\CRg$. An irreducible variety $V$ is an algebraic variety
which cannot be decomposed into an union of algebraic varieties
distinct from $V$.

The Hilbert Nullstellensatz states that the vanishing ideal
$\cI(V)=\{p \in \CRg\mid \forall \xi \in V, p(\xi)=0\}$ of an
algebraic variety $V\subset \C^{n}$ is the radical
$$
\sqrt{I}=\{p \in \CRg\mid \exists m \in \NN,\, p^{m}\in
I\}
$$
(see e.g. \cite{cox_ideals_2015}). This implies that $\sqrt{I}=\cap_{i=1}^{l}
\p_{i}$. We say that $I$ is radical if $I=\sqrt{I}$.

Considering now equations $\vb f=\{f_{1}, \ldots, f_{s}\} \subset \RRg$ with real
coefficients, the real variety defined by $\vb f$ is $V_{\RR} =
\cV_{\CC}(I) \cap \RR^{n} =
\cV_{\RR}(I)=\{\xi \in \RR^{n}\mid f_{i}(\xi)=0, i=1,\ldots,s\}$.
The vanishing ideal of $V_{\RR}$ is $\cI(V_{\RR})=\{p \in \RRg\mid
\forall \xi \in V_{\RR},\, p(\xi)=0\}$.
Let $\Sigma^{2}=\{ \sum_{j} p_{j}^{2},\, p_{j}\in \RRg\}$ be the sums of
squares of polynomials of $\RRg$.
The real Nullstellensatz states that $\cI(\cV_{\RR}(I))$ is the real radical of
$I$, defined as:
$$
\sqrt[\RR]{I}=\{p \in \RRg\mid \exists m \in \NN, s \in \Sigma^{2} \textup{ s.t. } p^{2m}+ s \in I\}
$$
(see e.g. \cite[p.~26]{marshall_positive_2008},
\cite[p.~85]{bochnak_real_1998}). If $I =\sqrt[\R]{I}$ then we say that $I$ is a \emph{real} or \emph{real radical} ideal.
The real radical of $I$ contains $\sqrt{I}$ and is the intersection of real
prime ideals $\p_{i}$ in $\RRg$ containing $I$, corresponding to the real irreducible
components of $\cV_{\R}(I)$. The example $f=
x_{1}^{2}+x_{2}^{2}$ such that $I=(f)=\sqrt{I}$ and
$\sqrt[\RR]{I}=(x_{1},x_{2})$ shows that the radical and real radical
ideals can define algebraic varieties of different dimensions.

Sets $S=\{\xi \in \RR^{n}\mid f_{1}(\xi)=0,\ldots, f_{s}(\xi)=0,
g_{1}(\xi)\ge 0, \ldots,$ $g_{r}(\xi)\ge 0 \}$ with $f_{i},g_{j}\in
\RRg$ are called basic semi-algebraic sets.
The real Nullstellensatz for $S$ states that the vanishing ideal $\cI(S)$ is
the $S$-radical of $I=(\vb f)$:
$$
\sqrt[S]{I}=\{p \in \RRg\mid \exists m \in \NN, (s_{\alpha}) \in
 (\Sigma^{2})^{\{0,1\}^{r}}\textup{ s.t. } p^{2m}+
\sum_{\alpha} s_{\alpha} \vb g^{\alpha} \in I\}
$$
(see e.g. \cite[th. 2.2.1]{marshall_positive_2008},
\cite[cor. 4.4.3]{bochnak_real_1998}, \cite{krivine:hal-00165658}, \cite{Stenglenullstellensatzpositivstellensatzsemialgebraic1974}).
The $S$-radical $\sqrt[S]{I}$ is related to the real radical of
an extended ideal $I_{S}$ defined by introducing slack variables
$s_{1}, \ldots, s_{r}$ for each non-negativity constraint defining
$S$: ${I_{S}}= (f_{1}, \ldots, f_{s}, g_{1}-s_{1}^{2}, \ldots,
g_{r}-s_{r}^{2}) \subset \RR[x_{1},\ldots, x_{n},$ $s_{1},$ $\ldots, s_{r}].$
Namely, we have $\sqrt[S]{I}= \sqrt[\RR]{I_{S}}\cap \RRg$ (by the Real Nullstellensatz, see e.g. \cite[p.~91]{bochnak_real_1998}).
Therefore, in the following we will focus on the computation of the real
radical of $I=(\vb f)$ and apply this transformation for the computation of $S$-radicals.

To describe the irreducible components of a variety
$\cV_{\CC}(f_{1},\ldots, f_{s})$ defined by equations $f_{1}, \ldots,
f_{s}\in \RRg$, we use tools from Numerical Algebraic Geometry, namely
a description of irreducible components by witness sets.
A witness set of an irreducible algebraic variety $V\subset \CC^{n}$ is a triple
$W = (\vb f, L, S)$ where $\vb f\subset \cI(V)$, $L$ is a generic linear
space of dimension $n-\dim(V)$ given by $\dim(V)$ linear equations and
$S=L\cap V \subset \CC^{n}$ is a finite set of $\deg(V)$ points.
Given equations $\vb f=\{f_{1},\ldots, f_{s}\}\subset \RRg$, a numerical
irreducible decomposition of $\cV_{\CC}(\vb f)$ can be computed as a
collection of witness sets $W_{i}=(\vb h_{i}, L_{i}, S_{i})$ such that
each irreducible component $V_{i}$ of $V$ is described by one and only
one witness set $W_{i}$ and all sample sets $S_{i}$ are pairewise disjoint.
Several methods, based on homotopy techniques, have been developed
over the past to compute such decomposition. See
e.g.
\cite{SommeseNumericalDecompositionSolution2001,
HauensteinRegenerativecascadehomotopies2011,
BatesNumericallySolvingPolynomial2013}.

The witness set $W$ of an (irreducible) algebraic variety $V$,
can be used to compute defining equations $\vb h=\{ h_{1},\ldots,h_{n}\}\subset \CRg$
such that $\cV_{\C}(\vb h)= V$. Homotopy techniques are employed to
generate enough sample points on $V$. The equations
$h_{i}$ are then computed by projection of the sample points onto $\le n+1$ generic
linear spaces of dimension $(\dim(V)+1)$ and by interpolation. See
e.g. \cite{SommeseNumericalDecompositionSolution2001}, for more details.

The numerical irreducible decomposition of $\cV_{\C}(\vb f)$ as a collection of witness
sets provides a description of all the irreducible components $V_{i}$
associated to the isolated primary components $Q_{i}$ of $I=(\vb f)$
\cite{atiyah_introduction_1994}.
To check that these primary components are reduced and thus prime
(i.e. $\sqrt{Q_{i}}=Q_{i}$), it is
enough to check that the Jacobian of $\bf f$
is of rank $n-\dim V_{i}$ (Jacobian criterion) at one of the sample points of the
witness set $W_{i}$, describing the irreducible component $V_{i}=\cV(P_{i})$.

Checking that $I=(\vb f)$ has no embedded component can also be done
by numerical irreducible decomposition of deflated ideals, as
described in \cite{KroneNumericalalgorithmsdetecting2017}. We are not
going to use this deflation technique to check non-embedded
components.

\vspace{-3mm}
\section{Moment relaxations}\label{sec:mop}

\subsection{Linear functionals}

We describe the dual of polynomial rings (see for instance
\cite{mourrain_polynomialexponential_2018} for more details). For
$\sigma \in (\RRg)^* =\hom_{\R}(\RRg,\R)= \{ \, \sigma \colon \RRg \to \R \mid \sigma \text{ is $\R$-linear} \, \}$, we denote $\val{\sigma}{f} = \sigma (f)$ the application of $\sigma$ to $f \in \RRg$. Recall that $(\RRg)^*\cong \R[[\vb{y}]] \coloneqq \R[[y_1,\dots,y_n]]$, with the isomorphism given by:
$\sigma \mapsto \sum_{\alpha \in \N^n} \val{\sigma}{\vb x^{\alpha}} \frac{\vb y^{\alpha}}{\alpha !}$,
where $\{\frac{\vb y^{\alpha}}{\alpha !}\}$ is the dual basis to $\{\vb x^{\alpha} \}$, i.e. $\val{\vb y^{\alpha}}{\vb x^{\beta}}=\alpha !\, \delta_{\alpha,\beta}$. With this basis we can also identify $\sigma \in (\RRg)^*$ with its sequence of coefficients $(\sigma_{\alpha})_{\alpha}$, where $\sigma_{\alpha}\coloneqq \val{\sigma}{\vb x^{\alpha}}$.

If $\sigma \in (\RRg)^*$ and $g \in \RRg$, we define the
\emph{convolution of $g$ and $\sigma$} as $g \star \sigma \coloneqq
\sigma \circ m_g \in (\RRg)^*$ where $m_{g}$ is the operator of multiplication
 by $g$ on the polynomials (i.e. $\val{g \star \sigma}{f} =
 \val{\sigma}{gf} \ \forall f$).
The operation $\star$ defines an $\RRg$-module structure on $\R[[\vb y]]$.
We define the \emph{Hankel operator} $H_{\sigma} \colon \RRg \to
(\RRg)^*, \ g \mapsto g \star \sigma$ and the {\em  annihilator}
$\ann(\sigma)= \ker H_{\sigma}$:  $g \in \ann(\sigma) \iff
H_{\sigma}(g)=0 \iff g \star \sigma = 0$.

We describe these operations in coordinates. If $\sigma = (\sigma_{\alpha})_{\alpha}$ and $g=\sum_{\alpha}g_{\alpha}\vb x^{\alpha}$ then $g \star \sigma = (\sum_{\beta}g_{\beta}\sigma_{\alpha+\beta})_{\alpha}$;
 the matrix $H_{\sigma}$ in the basis $\{\vb x^{\alpha} \}$ and  $\{\frac{\vb y^{\alpha}}{\alpha !}\}$ is $H_{\sigma}=(\sigma_{\alpha+\beta})_{\alpha,\beta}$.

\subsection{Truncation}

We introduce the same operations in a finite dimensional setting, considering only polynomials of bounded degree. If $A \subset \RRg$, $A_d\coloneqq \{ \,f \in A \mid \deg f \le d \, \}$. In particular $\RRg_d$ is the vector space of polynomials of degree $\le d$.
If $\sigma \in (\RRg)^*$ (resp. $\sigma \in (\RRg_r)^*, r\ge t$) then $\sigma^{[t]} \in (\RRg_t)^*$ denotes
its restriction to $\RRg_t$; moreover if $B\subset (\RRg)^*$ (resp. $B \subset (\RRg_r)^*, \ r\ge t$) then $B^{[t]} \coloneqq \{ \, \sigma^{[t]} \in (\RRg_t)^* \mid \sigma \in B \, \}$ .

If $\sigma \in (\RRg_t)^*$ and $g \in \RRg_t$, then $g \star \sigma \coloneqq \sigma \circ m_g \in (\RRg_{t-\deg g})^*$. If $\sigma \in (\RRg)^*$ (or $\sigma \in (\RRg_{r})^*$, $r\ge 2t$),
then we define $H_{\sigma}^t \colon \RRg_t \to (\RRg_t)^*, \ g \mapsto (g \star \sigma)^{[t]}$. We have $(g \star \sigma)^{[2t]} = 0 \iff H_{g \star \sigma}^{t}=0$: in analogy to the infinite dimensional setting we define $\ann_d(\sigma) \coloneqq \ker H_\sigma^d$.

For $\vb h = h_1,\dots,h_r \subset \RRg$ we define $\langle \vb h \rangle_t \coloneqq \big\{ \, \sum_{i=1}^r f_i h_i \in \R[\vb{X}]_t \mid f_i \in \R[\vb{X}]_{t-\deg h_i} \, \big\}$, the elements of $(\vb h)_t$ generated in degree $\le t$.

\subsection{Positive linear functionals and generic elements}

Let $A \subset \RRg$ (resp. $A \subset \RRg_t$). We define $A^{\perp} \coloneqq \big\{\, \sigma \in  (\RRg)^* \mid \val{\sigma}{f}=0 \ \forall f \in A \, \big\}$ (resp. $A^{\perp} \coloneqq \big\{\, \sigma \in  (\RRg_t)^* \mid \val{\sigma}{f}=0 \ \forall f \in A \, \big\}$).
Notice that $\sigma \in \tigen{\vb h}{t}^{\perp}$ (resp. $(\vb h)^{\perp}$) if and only if $(h \star \sigma)^{[t-\deg h]} = 0 \ \forall h \in \vb h$ (resp. $h \star \sigma = 0 \ \forall h \in \vb h$).

We say that $\sigma\in (\RRg_{2t})^*$ is \emph{positive semidefinite} (psd) $\iff H_{\sigma}^t$ is psd, i.e. $\val{H_{\sigma}^t(f)}{f}=\val{\sigma}{f^2}\ge 0 \ \forall f\in \RRg_t$. If $\sigma$ is pds then $\val{\sigma}{f^2}=0 \Rightarrow f \in \ann_t(\sigma)$ \cite[3.12]{lasserre_moment_2013}

For $G \subset \RRg_t$ we finally define the closed convex cone:
\[
    \cL_t(G) \coloneqq \{ \, \sigma \in (\RRg_t)^* \mid \sigma \text{ is psd and } \forall g \in (G\cdot\Sigma^{2})_{t} \  \val{\sigma}{g}\ge 0 \, \},
\]
see \cite{baldi:hal-03082531} for more details. In particular $\cL_t(\pm \vb h) = \{ \, \sigma \in \tigen{\vb h}{t}^{\perp} \mid \sigma \text{ is psd} \, \}$. We use $\cL(G)$ for the infinite dimensional case. Notice that $\cL(G)^{[t]}\subset \cL_t(G)\ \forall t$.

Linear functionals of special importance are evaluations $\eval_{\xi}$ defined as $\val{\eval_{\xi}}{f} = f(\xi)$. For $\xi \in \cV_{\R}(\vb h)$ we have $\eval_{\xi} \in \cL(\pm \vb h)$.

\begin{definition}
\label{def::generic_t}
    We say that $\sigma^* \in \cL_t(\pm \vb h)$ is \emph{generic} if $\rank H_{\sigma^*}^t=\max \{ \rank H_{\eta}^t \mid \eta \in \cL_t(\pm \vb h) \}$.
\end{definition}
Genericity is characterized as follows, see e.g. \cite[prop. 4.7]{lasserre_moment_2013}:
\begin{proposition}
    \label{prop::genericity}
    Let $\sigma \in \cL_{2t}(\pm \vb h)$. The following are equivalent:
    \begin{enumerate}
        \item $\sigma$ is generic;
        \item $\ann_t(\sigma) \subset \ann_t(\eta) \ \forall \eta \in \cL_{2t}(\vb g)$;
        \item $\forall d\le t$, we have: $\rank H_{\sigma}^d=\max \{ \rank H_{\eta}^d \mid \eta \in \cL_{2t}(\pm \vb h) \}$.
    \end{enumerate}
\end{proposition}

Generic elements can be used to compute the real radical of ideals, see \cite[th. 7.39]{rostalski_algebraic_2009}. We give in \Cref{thm:realradical} a proof of this result. See also \cite[th. 3.16]{baldi:hal-03082531} for a generalisation to quadratic modules.

\begin{theorem} \label{thm:realradical}
  Let $\sigma^* \in \cL_{2d}(\pm\vb h)$ be generic and $I=(\vb h)$. Then for every $d \ge \deg \vb h$ we have $I \subset (\ann_d(\sigma^*))\subset \sqrt[\RR]{I}$. Moreover for $d$ big enough $(\ann_d(\sigma^*)) = \sqrt[\RR]{I}$.
\end{theorem}
\begin{proof}
  The inclusion $I \subset (\ann_d(\sigma^*))$ is clear since $\vb h \subset \ann_d(\sigma^*)$ by definition. Now let $J = \sqrt[\RR]{I}$. Notice that, for $\xi \in \RR^n$, $\ann_d(\eval_{\xi}) = \cI(\xi)_d = (x_1-\xi_1,\dots,x_n-\xi_n)_d$. Moreover, if $\xi \in \cV_{\RR}(I)$, then $\eval_{\xi}^{[2d]}\in \cL_{2d}(\pm \vb h)$. Then, since $\sigma^*$ is generic:
  \[
    \ann_d(\sigma^*) \subset \bigcap_{\xi \in \cV_{\RR}(I)} \ann_d(\eval_{\xi})=\bigcap_{\xi \in \cV_{\RR}(I)} \cI(\xi)_d = J_d,
  \]
  and thus $(\ann_d(\sigma^*))\subset J$.

  For the second part, let $g_1, \dots, g_k$ be generators of
  $J$. By the Real Nullstellensatz, $\forall i$ there exists $m_i \in \N,
  s_i \in \Sigma^2$ such that $g_i^{2^{m_i}}+s_i \in I$. Then for $d$
  big enough and $\sigma \in \cL_{2d}(\pm \vb h)$ we have
  $\val{\sigma^{[2d]}}{g_i^{2^{m_i}}+s_i} = 0$, thus
  $\val{\sigma^{[2d]}}{g_i^{2^{m_i}}} = 0$ and $g_i \in \ann_d(\sigma)$. This implies $J \subset (\ann_d(\sigma))$ for all $\sigma \in \cL_{2d}(\pm \vb h)$, and in particular for $\sigma = \sigma^*$ generic.
\end{proof}

The goal of the paper is to find an effective algorithm, based on
\Cref{thm:realradical}, to compute $\sqrt[\RR]{I}$. In the case of a
finite real variety, the flat extension criterion
\cite{lasserre_semidefinite_2008,lasserre_moment_2013} certifies that
$(\ann_d(\sigma^*)) = \sqrt[\RR]{I}$ for some $d \in \N$. We will
focus in the positive dimensional case, when such a criterion cannot apply.

\subsection{Polynomial Optimization and Exactness}

Let $f, \vb g \in \RRg$. The goal of Polynomial Optimization is to find:
\begin{equation}\label{eq:pop}
      f^* \coloneqq \inf \ \big\{ \, f(x)\in \R \mid x \in \R^n, \ g_i(x) \ge 0 \ \textup{ for } i=1, \ldots,s \,\big\}.
    \end{equation}
    that is the infimum $f^{*}$
of the \emph{objective function} $f$ on the \emph{basic semialgebraic set} $S \coloneqq \{ \, x \in \R^n \mid \ g_i(x) \ge 0 \ \textup{ for } i=1, \ldots,s \, \} $. In particular we will consider the case of equalities $h_i = 0$, obtained as $\pm h_i \ge 0$. To solve problem \eqref{eq:pop}
Lasserre \cite{lasserre_global_2001} proposed to use two hierarchies of finite dimensional convex cones depending on an order $d\in \N$. We describe the Moment Matrix hierarchy and the property of \emph{exactness}, see \cite{baldi:hal-03082531} for more details.

\begin{definition}
We define the \emph{MoM relaxation of order $d$} of problem \eqref{eq:pop} as $\cL_{2d}(\vb g)$ and the infimum:
\begin{equation}
\label{def::momrel}
  f^*_{\mom,d}  \coloneqq \inf \big\{ \, \val{\sigma}{f} \in \R \mid \sigma \in \cL_{2d}(\vb g), \ \val{\sigma}{1} = 1 \,\big\}.
\end{equation}
\end{definition}

We will call Problem \eqref{def::momrel} a \emph{Moment Optimization
  Problem} (MOP). It can be efficiently solved by semidefinite
programming, using interior point methods. Taking $f=1$, these methods
yield an interior point of $\cL_{2d}(\vb g)$, that is a {\em generic}
element $\sigma^{*}$ in $\cL_{2d}(\vb g)$.

Usually we are interested in minimizers of $f$ with bounded norm,
i.e. minimizers in some closed ball defined by $r-\norm{\vb x}^2 \ge
0$ (\emph{Archimedean condition}). If the Archimedean condition and
some regularity conditions at the minimizers of $f$ hold (known as
\emph{Boundary Hessian Conditions} or BHC), the MoM relaxation is
\emph{exact}: for some $d\in \N$ the minimum is reached, i.e. $f^* =
f^*_{\mom,d}$, and we can effectively recover the minimizers (see
\cite[th. 4.8]{baldi:hal-03082531}). Using the flat extention
criterion for the Hankel matrix $H_{\sigma}^d$ (associated to a
minimizing moment sequence $\sigma$) we can effectively test
exactness. As BHC hold generically, exactness is also generic (see
\cite[cor. 4.9]{baldi:hal-03082531}).

\section{Orthogonal polynomials and annihilator}\label{sec:orthogonal}

To compute the real radical, we need to compute a basis of the
annihilator of a truncated positive linear functional $\sigma \in (\RRg_{2d})^{*}$ such that
$\val{\sigma}{p^{2}}\ge 0$ for $p\in \RRg_{d}$.
In this section, we describe an
efficient algorithm to compute a basis of $\ann_{d}(\sigma)=\{ p\in \RRg_{d} \mid
p \star \sigma =0\}= \{p \in \RRg_{d}\mid \val{\sigma}{p^{2}}=0\}$.
It is a Gram-Schmidt orthogonalization process, using
the inner product $\langle \cdummy, \cdummy
\rangle_{\sigma}$ defined, for $p,q\in \RRg_{d}$, by
$$
\val{p}{q}_{\sigma} \assign \val{\sigma}{p\,q}.
$$

By ordering the monomials basis of $\RRg_{d}$ and projecting
successively a monomial $\vb x^{\alpha}$ onto the space spanned by the previous
monomials, we construct monomial basis $\vb b = \{ \vb{x}^{\beta} \}$ of
$\RRg_{d}/\ann_{d}(\sigma)$, a corresponding basis of orthogonal
polynomials $\vb{p}= (p_{\beta})$ and a basis $\vb{k}= (k_{\gamma})$
of $\ann_{d}(\sigma)$. The orthogonal polynomials are such that
\[ \langle p_{\beta}, p_{\beta'} \rangle_{\sigma} = \left\{ \begin{array}{ll}
     >0 & \textup{if } \beta = \beta'\\
     0  & \textup{otherwise},
\end{array} \right. \]
and for all $\beta, \gamma$, we have $\val{p_{\beta}}{k_{\gamma}}_{\sigma}=\val{k_{\gamma}}{k_{\gamma}}_{\sigma}=0$.

To compute these polynomials, we use a projection defined on the
orthogonal of the space
spanned by orthogonal polynomials $\vb{p}= [p_1, \ldots, p_l]$ such
that $\val{p_{i}}{p_{i}}_{\sigma}>0$ and
$\val{p_{i}}{p_{j}}_{\sigma}=0$ if $i\neq j$, as follows: for $f \in \RRg_{d}$,
$$
\proj(f,\vb p)= f - \sum_{i=1}^{l} {\val{f}{p_{i}}_{\sigma}\over
  \val{p_{i}}{p_{i}}_{\sigma}} \, p_{i}.
$$
By construction, we have $\val{\proj(f,\vb p)}{p_{i}}_{\sigma}=0$ for $i=1,\ldots,l$.
In practice, the implementation of this projection is done by the so-called
Modified Gram-\-Schmidt projection algorithm, which is known to have a
better numerical behavior than the direct Gram-Schmidt
orthogonalization process {\cite{trefethen_numerical_1997}}[Lecture 8].

To compute a basis of $\ann_{d}(\sigma)$, we choose a monomial
ordering $\prec$ compatible with the degree (e.g. the graded
reverse lexicographic ordering) and build the list of monomials $\vb s$
of degree $\le d$ in increasing order for this ordering $\prec$.
Algorithm \ref{algo:orthogonal} chooses incrementally a new monomial in the list
$\vb s$ and projects it on the space spanned by the previous
orthogonal polynomials. The new monomials computed by the function
$\textup{next}(\vb s,\vb b, \vb l)$ are the monomials with the lowest degree in $\vb s$, ordered w.r.t. $\prec$, not
in $\vb b$ and not divisible by a monomial of $\vb l$.

{\begin{algorithm}\caption{\label{algo:orthogonal}Orthogonal polynomials
      and annihilator of $\sigma$}
{\strong{Input:} a positive linear functional $\sigma\in
      \RRg_{2d+2}^{*}$.

\begin{itemize}
  \item Let $\vb{b} \assign []$; $\vb{p} \assign []$; $\vb{k} \assign []$;
  $\vb{l}= []$; $\vb{n} \assign [1]$; $\vb{s}
  \assign [\vb x^{\alpha},  |\alpha|\le d]$;

  \item while $\vb{n} \neq \emptyset$ do
  \begin{itemize}
    \item for each $\vb x^\alpha \in \vb{n}$,
      \begin{enumerate}
      \item $p_{\alpha} \assign \proj (\vb{x}^{\alpha}, \vb p)$;

      \item compute $v_{\alpha}=\val{p_{\alpha}}{p_{\alpha}}_{\sigma}$;
      \item if $v_{\alpha}\neq 0$ then

      \ \ add $\vb x^{\alpha}$ to $\vb{b}$;
      \ add $p_{\alpha}$ to $\vb{p}$;

      else

      \ \ add $k_{\alpha} := p_{\alpha}$ to $\vb k$;
      \ add $\vb x^\alpha$ to $\vb{l}$;

      end;
    \end{enumerate}
    \item $\vb{n} \assign \textup{next} (\vb{s}, \vb{b}, \vb{l}) ;$
  \end{itemize}
\end{itemize}
{\strong{Output:}}
\begin{itemize}
  \item a basis  $\vb{k}= [k_{\gamma}]_{\vb x^{\gamma} \in \vb{l}}$ of the annihilator $\ann_{d}(\sigma)$ and their leading monomials $\vb l=[\vb x^{\gamma}]$;
  \item a basis of orthogonal polynomials $\vb{p}= [p_{\beta_i}]$;
  \item a monomial set $\vb{b}= [\vb x^{\beta_1}, \ldots, \vb x^{\beta_r}]$.
\end{itemize}}
\end{algorithm}}

By construction, the vector space spanned by $\vb b$ and $\vb p$ are
equal at each loop of the algorithm.
As the function $\textup{next}(\vb s,\vb b, \vb l)$ outputs
monomials in $\vb s$ greater than $\vb b$ then the monomials in $\vb n$ are greater than the
monomials in $\vb b$.
Thus, the leading term of $k_{\gamma}\in \vb k$ is $\vb x^{\gamma}$.

Let $\vb k$, $\vb l$, $\vb p$, $\vb b$  denote the output of Algorithm \ref{algo:orthogonal}. For $\alpha\in \NN^{n}$, let $(\vb k)_{\preceq \alpha}$ be the
vector space spanned by the elements of the form $\vb x^{\delta}
k_{\gamma}$ with $\delta+\gamma \preceq \alpha$. Similarly, $\vb
p_{\preceq \alpha}$ is the set of $p_{\beta} \in \vb p$ such that
$\beta \preceq \alpha$.
We prove that $\vb k$ is a Grobner basis of $\ann_{d}(\sigma)$, that
is any element of $\ann_{d}(\sigma)$ reduces to $0$ by $\vb k$:

\begin{proposition} Let $\sigma\in
      \RRg_{2d+2}^{*}$, $\vb k, \vb p$ be the output of \Cref{algo:orthogonal}.
For $\vb x^{\alpha}\in (\vb l)_{d}$, i.e.  divisible by a monomial in $\vb
l$ and of degree $|\alpha|\le d$, $p_{\alpha}= \proj(\vb x^{\alpha},
\vb p_{\preceq \alpha})$ is in $(\vb k)_{\preceq \alpha} \subset \ann_{d}(\sigma)$.
\end{proposition}
\begin{proof}
Let us prove it by induction on the ordering of $\alpha$. The lowest
element in $(\vb l)_{d}$ is a monomial $\vb x^{\gamma}$ of $\vb
l$. As $k_{\gamma}= \proj(\vb x^{\gamma}, \vb p_{\preceq  \gamma})$ is such that
$\val{k_{\gamma}}{k_{\gamma}}_{\sigma}=\val{\sigma}{k_{\gamma}^{2}}=0$,
$k_{\gamma} = \proj(\vb x^{\gamma},\vb p_{\preceq  \gamma}) \in
(\vb k)_{\preceq \gamma} \subset \ann_{d}(\sigma)$.
Then the induction hypothesis is true for
the lowest monomial of $(\vb l)_{d}$.

Assume that it is true for $\vb x^{\alpha'} \in (\vb l)_{d}$ and for all the smaller monomials w.r.t. $\prec$. Let
$\vb x ^{\alpha}$ be the next monomial in $(\vb l)_{d}$ for the monomial ordering
$\prec$.
Then, there exists $\vb x^{\alpha''}\in (\vb l)_{\preceq \alpha'}$ and $i_{0}\in
1,\ldots,n$ such that $x_{i_{0}} \vb x^{\alpha''} = \vb
x^{\alpha}$. As $p_{\alpha} - x_{i_{0}} p_{\alpha''}$ has a leading
term smaller that $\vb x^{\alpha}$, it can be written as a linear combination
of $p_{\alpha'}= \proj(\vb x^{\alpha'}, \vb p_{\prec \alpha'})$ with $\alpha'\prec \alpha$.
More precisely, we have
\[
p_{\alpha} = x_{i_{0}} p_{\alpha''} + \sum _{\delta \prec \alpha, \vb
x^{\delta} \in (\vb l)_{d}} \lambda_{\delta}\, p_{\delta}  +
\sum_{\beta \prec \alpha, \vb x^{\beta}\in \vb b} \mu_{\beta}\, p_{\beta},
\]
for some $\lambda_{\delta}, \mu_{\beta} \in \R$.

By induction hypothesis, $p_{\alpha''}, p_{\delta}\in (\vb
k)_{\preceq\alpha'}\subset (\vb k)_{\preceq\alpha}\subset
\ann_{d}(\sigma)$. Moreover, as $p_{\alpha''}\in
\ann_{d}(\sigma)\subset \ann_{d+1}(\sigma)$, for any $p\in \RRg_{d}$
we have
$\val{x_{i_{0}} p_{\alpha''}}{p}_{\sigma} =
\val{p_{\alpha''}}{x_{i_{0}}\, p}_{\sigma}=0$.
This shows that $x_{i_{0}} p_{\alpha''} \in (\vb k)_{\preceq \alpha} \cap \ann_{d}(\sigma)$.

By definition of $p_{\alpha}=\proj(\vb x^{\alpha}, \vb p_{\prec\alpha})$, $\val{p_{\alpha}}{p_{\beta}}_{\sigma}=0$ for
$\vb x^{\beta}\in \vb b_{\prec\alpha}$ so that
$
\mu_{\beta} = {\val{p_{\alpha}}{p_{\beta}}_{\sigma} \over
  \val{p_{\beta}}{p_{\beta}}_{\sigma}} = 0
$
and $p_{\alpha}\in (\vb k)_{\preceq \alpha} \cap \ann_{d}(\sigma)$.

As $(\vb k)_{\preceq \alpha}=(\vb k)_{\preceq
  \alpha'}+\vspan{p_{\alpha}}$, we have $(\vb k)_{\preceq \alpha}\subset
\ann_d(\sigma)$, which proves the induction hypothesis for $\alpha$ and
concludes the proof.
\end{proof}

This proposition explains why the function $\textup{next}(\vb s,\vb b, \vb l)$
only outputs the monomials with the lowest degree in $\vb s$, ordered w.r.t. $\prec$, not
in $\vb b$ and not divisible by a monomial of $\vb l$.

This algorithm is an optimization of Algorithm 4.1 in
{\cite{mourrain_polynomialexponential_2018}} or
Algorithm 3.2 in \cite{mourrain:hal-01515366}. It strongly exploits
the positivity of the linear functional $\sigma$ and improves
significantly the performance. We will illustrate its behavior in Section \ref{sec:example}.

\begin{remark}\label{rem:flatext}
When the real variety $\cV_{\RR}(\vb f)$ is finite, the flat extension
test on the rank of $H_{\sigma}^{k}$ can be replaced by testing that
the set $\vb l$ of initial terms contains a power of each variable
$x_{i}$. This is equivalent to the fact that $\R[\vb x]/(\vb k)$ is finite
dimensional or equivalently that the rank of $H_{\sigma}^{d}$ is constant for
$d\gg 0$.
\end{remark}

\section{Real irreducible components}\label{sec:isreal}
We introduce an effective algorithm for testing real radicality in the irreducible case.
\subsection{Genericity}
Let $\C^N$ be the $N$-dimensional affine space and $\C[t_1,\dots,t_N] = \C[\vb t]$ be its coordinate (polynomial) ring. We say that a property holds \emph{generically} in $\C^N$ if there exists finitely many nonzero polynomials
$\phi_1,\dots,\phi_l \in \C[\vb t]$ such that, for $\xi \in \C^N$,
when $\phi_1(\xi)\neq 0,\dots,\phi_l(\xi)\neq 0$ the
property holds for $\xi$.

In particular we will consider linear maps $A \in \hom_{\C}(\C^n,\C^{k+1})$ as elements in $\C^{n(k+1)}$ in the natural way, and thus talk about \emph{generic linear maps}.

\subsection{Smooth Complex and Real Zeros}
We recall the definition of \emph{smooth zero}. We refer to
\cite{shafarevich_basic_2013-1} for the complex case and to
\cite{marshall_positive_2008} for the real case.

We say that a variety $V\subset \C^{n}$ is {\em  defined over $\R$}, if
$\cI{(V)}$ is generated by a family of polynomials with coefficient in $\R$. For $A\subset \C^n$ we denote by $\cl{A}$ its Zariski closure.

Hereafter $\K$ denotes a field of characteristic $0$ and $\overline{\K}$ its algebraic closure.

\begin{definition}
    Let $I = (f_1,\dots ,f_m) \subset \K[\vb x]$ be a prime ideal and $V = \cV_{\overline{\K}}(I)$. We say that $\xi \in \cV_{\K}(I)$ is a \emph{smooth zero} of $I$ if $\rank \jac (f_1,\dots,f_m)(\xi) = n-\dim V$.
\end{definition}
For $\K = \C$ the mapping $V \mapsto \cI_{\C}(V)$ is a bijection between irreducible varieties in $\C^n$ and prime ideals. Moreover, for a prime ideal $I$, smooth zeros of $I$ and smooth points of $\cV_{\C}(I)$ coincide, and they are dense. On the other hand for $\K = \R$ the mapping $V \mapsto \cI_{\R}(V)$ is a bijection between irreducible varieties in $\R^n$ and prime ideals which are real radical. For prime ideals $I$ which are not real radical, smooth zeros of $I$ are not dense in $\cV_{\R}(I)$.

\begin{example} Here are examples of reducible
  and irreducible algebraic varieties with dense complex smooth points but with no real smooth point.
    \begin{itemize}
        \item $I=(x^2+y^2)\subset\R[x,y]$ is a prime, non real radical ideal, as $\cV_{\R}(I) = \{(0,0)\}$ and $\sqrt[\R]{I} = (x, y)$. $I$ does not have smooth real zeros. Notice that $(x^2+y^2)\subset\C[x,y]$ is not prime, since $x^2+y^2 = (x+iy)(x-iy)$.
        \item $I=(x^2+y^2+z^2)\subset\R[x,y,z]$ is a prime, non real radical ideal, as $\cV_{\R}(I) = \{(0,0,0)\}$ and $\sqrt[\R]{I} = (x, y, z)$. $I$ does not have smooth real zeros. In this case $(x^2+y^2+z^2)\subset\C[x,y,z]$ is prime, since $x^2+y^2+z^2$ is irreducible over $\C$.
    \end{itemize}
\end{example}

We recall criterions for testing whether a prime ideal $I \subset \RRg$ is real radical or not.

\begin{theorem}[{Simple Point Criterion \cite[th. 12.6.1]{marshall_positive_2008}}]
\label{thm::simple_point}
  Let $I$ be a prime ideal of $\RRg$. The following are equivalent:
  \begin{itemize}
    \item $I$ is a real radical ideal;
    \item $I = \cI(\cV_{\R}(I))$;
    \item $\cl{\cV_{\R}(I)} = \cV_{\C}(I)$;
    \item $I$ has a smooth real zero.
  \end{itemize}
\end{theorem}

\begin{definition}
  If $V \subset \CC^n$ then $V_{\RR}$ denotes the real points of $V$, i.e. $V_{\RR} = V \cap \conj{V} = V \cap \R^n$.
\end{definition}

Let $V \subset \C^n$ be an irreducible variety defined over $\R$ and $I \subset \RRg$ the ideal defined by its real generators. If follows from \Cref{thm::simple_point} that $V_{\RR} = \cV_{\R}(I)$ is Zariski dense in $V$ if and only if $I$ is a real radical ideal. In this case we say that $V$ is \emph{real}.

For hypersurfaces there exists another criterion based on the change of sign of the defining polynomial.
\begin{theorem}[{Sign Changing Criterion \cite[th. 12.7.1]{marshall_positive_2008}}]
\label{thm::sign_changing}
  Let $f \in \RRg$ be an irreducible polynomial. The following are equivalent:
  \begin{itemize}
    \item $(f)$ is a real radical ideal;
    \item $(f)$ has a smooth real point (i.e. there exists $\xi \in \cV_{\R}(I)$ such that $\grad{f}(\xi) \neq 0$);
    \item the polynomial $f$ changes sign in $\R^n$ (i.e. there exists $x,y\in \R^n$ such that $f(x)f(y)<0$).
  \end{itemize}
\end{theorem}

\subsection{Test for Real Radicality}

We reduce the problem of testing real radicality to the
hypersurface case, and then use the Simple Point Criterion. For that prupose
we project $V \subset \C^n$, irreducible variety of dimension
$k$, on a linear subspace $\C^{k+1} \subset \C^n$, in such a way $V$
and $\cl{\pi(V)}$ are \emph{birational}. (see
\cite[p.~38]{shafarevich_basic_2013-1} for the definition).

It is classical that every irreducible (affine) variety is birational to an hypersurface. We recall briefly this result to show that we can choose a generic projection as birational morphism, as done for the geometric resolution or rational representation, see for instance \cite{lecerf_computing_2003} or \cite{bostan_algorithmes_2017-1}.

\begin{lemma}
\label{lem::birational_proj}
  Let $V\subset \CC^n$ be an irreducible varierty of dimension $k$ and $\pi \colon \CC^n \to \CC^{k+1}$ be a generic projection. Then $V$ is birational to $\pi(V)$, i.e. $V \cong \cl{\pi(V)}$.
\end{lemma}
\begin{proof}{(sketch)}
  The birational morphism in \cite[p.~39]{shafarevich_basic_2013-1} can be given as a generic projection. Indeed we can choose algebraically independent elements $l_1, \dots, l_k$  generic linear forms in the indeterminates $\vb x$ (see for instance \cite[p.~488]{bostan_algorithmes_2017-1}). The choice of the primitive element $l_{k+1}$ is generic (see for instance \cite[th.~15.8.1]{artin_algebra_2017}: one can choose $l_{k+1}$ as a generic linear form). Then $l_1, \dots, l_{k+1}$ define the projection $\pi \colon\CC^n \to \CC^{k+1}, \ \xi \mapsto (l_1(\xi), \dots , l_{k+1}(\xi))$ and $V$ is birational to $\cl{\pi(V)}$.
\end{proof}

We choose a generic projection defined over $\R$. In this case we
show that $V$ has a smooth real point if and only if $\cl{\pi(V)}$ has
a smooth real point, using the following propositions.

\begin{proposition}
\label{prop::proj_smooth_point}
  Let $V\subset \CC^n$ be an irreducible varierty defined over $\R$ of dimension $k$, and let $\pi \colon \CC^n \to \CC^{k+1}$ be a generic projection defined over $\R$. Then $\cl{\pi(V)}$ is defined over $\R$ and if $V$ has a smooth real point then $\cl{\pi(V)}$ has a smooth real point.
\end{proposition}
\begin{proof}
    Let $\pi \colon \CC^n \to \CC^{k+1}$ be a generic projection defined
    over $\R$. As $V$ is defined over $\R$, $\cl{\pi(V)}$ is also defined over $\R$ since $\cI(\pi(V))$ is the elimination ideal $\left(\cI(V)+(\pi(\vb x) - \vb y)\right) \cap \R[\vb y]$, where $\vb y = y_1,\dots,y_{k+1}$ are coordinates of $\CC^{k+1}$ (see \cite{cox_ideals_2015}).

    If $V$ has a smooth real point then $V_{\R}$ is Zariski dense in $V$ by \Cref{thm::simple_point}. Then $\pi(V_{\R})$ is Zariski dense in $\pi(V)$. Since $\pi$ is defined over $\R$ we have that $\pi(V_{\R})\subset (\pi(V))_{\R}$ and $(\pi(V))_{\R}$ is Zariski dense in $\pi(V)$.
    Then $\cl{(\pi(V))_{\R}} = \cl{\pi(V)}$ and by \Cref{thm::simple_point} $\cl{\pi(V)}$ has a smooth real point.
\end{proof}
\begin{proposition}
\label{prop::prj_not_smooth}
  Let $V\subset \CC^n$ be an irreducible variety defined over $\R$ of
  dimension $k$ without smooth real points. Then, for a generic
  projection $\pi \colon \CC^n \to \CC^{k+1}$  defined over $\R$, $\cl{\pi(V)}$ is defined over $\R$ and has no smooth real points.
\end{proposition}
\begin{proof}
  By \Cref{prop::proj_smooth_point}, $\cl{\pi(V)}$ is defined over $\R$.

  Assume now that $\cl{\pi(V)}$ has a smooth real point. Since $V$ is
  generically birational to $\pi(V)$ (\Cref{lem::birational_proj}),
  the preimage of a generic smooth point in $\pi(V)$ is a single point in $V$, which is smooth. If $\pi$ is defined over $\R$ then
  this smooth point $p\in V$ is real since
  $\pi(p)=\overline{\pi(p)}=\pi(\overline{p})$ implies that
  $p=\overline{p}$, showing that $V$ has a smooth real point.
\end{proof}
\begin{proposition}
\label{prop::proj_is_complex}
  Let $V\subset \CC^n$ be an irreducible variety not defined over $\R$ of dimension $k$. If $\pi \colon \CC^n \to \CC^{k+1}$ is a generic projection defined over $\R$ then $\cl{\pi(V)}$ is not defined over $\R$.
\end{proposition}
\begin{proof}
  $V$ is not defined over $\R$ if and only if $V \neq \conj{V}$. Thus
  there exists $p \in V$ such that $\conj{p} \notin V$. Then for $\pi
  \colon \CC^n \to \CC^{k+1}$ a generic projection, we have
  $\pi(\conj{p})\notin \cl{\pi(V)}$ (see e.g. \cite[sec.~3]{blanco_computing_2004}).
  As $\pi$ is defined over $\R$, we have $\pi(p) \in \cl{\pi(V)}$ and
  $\conj{\pi(p)}=\pi(\conj{p})\notin \cl{\pi(V)}$. Therefore,
  $\cl{\pi(V)} \neq \conj{\cl{\pi(V)}}$ and $\cl{\pi(V)}$ is not defined over $\R$.
\end{proof}
\begin{theorem}
\label{thm::real_iff_real_proj}
  Let $V\subset \CC^n$ be an irreducible variety of dimension $k$. Then $V$ is defined over $\R$ and has a smooth real point if and only if, for $\pi \colon \CC^n \to \CC^{k+1}$ generic projection defined over $\R$, $\cl{\pi(V)}$ is defined over $\R$ and has a smooth real point.
\end{theorem}
\begin{proof}
  If $V$ has a smooth real point then we apply \Cref{prop::proj_smooth_point} to conclude that $\cl{\pi(V)}$ has a smooth real point. If $V$ is defined over $\R$ but has no smooth real point, we apply \Cref{prop::prj_not_smooth} and deduce that $\cl{\pi(V)}$ has no smooth real points. Finally, if $V$ is not defined over $\R$ we apply \Cref{prop::proj_is_complex} to show that $\cl{\pi(V)}$ is not defined over $\R$.
\end{proof}
\begin{corollary}
\label{cor::proj_change_sign}
  Let $V\subset \CC^n$ be an irreducible variety of dimension $k$, and $\pi \colon \CC^n \to \CC^{k+1}$ a generic projection defined over $\R$. Then the following are equivalent:
  \begin{enumerate}
    \item $V$ is defined over $\RR$ and the real generators of $\cI(V)$ define a real radical ideal in $\RRg$;
    \item $\cI(\pi(V))$ is generated by a real polynomial, irreducible over $\CC$, which changes sign in $\RR^{k+1}$.
  \end{enumerate}
\end{corollary}
\begin{proof}
  By \Cref{thm::simple_point}, real generators of $\cI(V)$ define a real radical ideal if and only if $V$ has a smooth real point . Then $(i) \iff (ii)$ follows from \Cref{thm::real_iff_real_proj} and \Cref{thm::sign_changing}.
\end{proof}

We finally describe the algorithm for testing real radicality.
{\begin{algorithm}[b]\caption{\label{algo:is real} Test real radicality}
{\strong{Input:} An irreducible variety $V \subset \CC^n$ of dim. $k$ and $\epsilon, r > 0$.
\begin{enumerate}
    \item Fix a generic projection $\pi \colon \C^n \to \C^{k+1}$;
    \item Compute the irreducible polynomial $h$ defining $\cl{\pi(V)}$;
    \item If $h$ is not real return false;
    \item Choose a generic point $\xi \in \R^{k+1}$ such that $h(\xi) \neq 0$;
    \item $s \assign \sign(h(\xi))$;
    \item Let $f = \norm{\vb x -\xi}^2$. Solve the MOP:
    \[
        f^*_{\mom,d} = \inf \{ \val{\sigma}{f} \mid \sigma \in \cL_{2d}(\pm (h+s\epsilon), r^2-f), \ \val{\sigma}{1} = 1 \};
    \]
    \item Extract a minimizer $\eta$ and check that $h(\xi)h(\eta)<0$.
\end{enumerate}
\strong{Output:} False if the MOP is not feasible, true if the MOP is feasible and $h(\xi)h(\eta)<0$.
}\end{algorithm}}

In step (i) we fix a generic real projection such that $V$ is birational to $\cl{\pi (V)}$ (\Cref{lem::birational_proj}).

In steps (ii) and (iii) we compute a minimal degree polynomial $h$ of
the hypersurface $\cl{\pi (V)}$, scaled so that one of its
coefficients is $1$ and stop if it has non real coefficients.

In steps (iv), (v) and (vi) we check if the real polynomial $h$
defines a real radical ideal, using \Cref{thm::sign_changing}. We find
$\xi \in \R^{k+1}$ where $h$ is not vanishing, and then search
another point where $h$ has opposite sign, by Moment Optimization.

If $h$ does not change sign then $\cV_{\R}(h+s\epsilon) = \emptyset$ and the MOP will not be feasible (see for instance \cite{lasserre_semidefinite_2008}).

On the other hand if $h$ changes sign there exist $\eta\in \R^{k+1}$ such that $h(\xi)h(\eta)<0$. If $\norm{\eta-\xi}<r$ and $0<\epsilon\le f(\eta)$ then the MOP has a solution. For generic $\xi$ the minimizer will be a unique smooth point, the MOP will be exact (since we added the ball constraint $r^2-f \ge 0$, the Archimedean property holds and generecally the MOM relaxation is exact), and we can certify that $h$ changes sign. The constraint $r^2-\norm{\vb x - \xi}^2 \ge 0$
is not necessary if $\cV_{\R}(h)$ is compact, since in this case the Archimedean hypothesis is already satisfied.

The correctness of \Cref{algo:is real} follows from \Cref{cor::proj_change_sign}.

\subsection{Test}
We test \Cref{algo:is real} for two simple cases, using the Julia packages \href{https://gitlab.inria.fr/AlgebraicGeometricModeling/MomentTools.jl}{\texttt{MomentTools.jl}} and \href{https://github.com/bmourrain/MultivariateSeries.jl}{\texttt{MultivariateSeries.jl}}.
\begin{example}
  We check that the irreducible polynomial $h=x^2+y^2 \in \R[x,y]$ defines an ideal $I = (h)$ that is not real radical. We randomly choose $\xi = (-1.5667884102749219, -0.5028780359864093)$, where $h(\xi)>0$. We check that $h$ does not change sign, detecting the infeasibility of the optimization problem.
  \begin{verbatim}
    X = @polyvar x y
    h = x^2 + y^2
    s = sign(h(X => xi))
    dist = sum((xi - vec(X)).^2)
    e = 0.01
    v, M = minimize(dist, [h+s*e], [9 - dist],
                                    X, 4, optimizer);
  \end{verbatim}
\vspace{-3mm}
  The termination status \texttt{termination\_status(M.model)}:
  \begin{verbatim}
INFEASIBLE::TerminationStatusCode = 2
  \end{verbatim}
\vspace{-3mm}
  of the optimization shows the infeasibility of the moment optimization program and
                                    that $I$ is not real radical.
\end{example}
  In the same way we detect the sign change. For $h=x^2 +
  y^2 - 1$ and $\xi$ as above, we find $\eta = (-0.9473807839956285,$ $-0.30408822493309284)$ and $h(\xi)h(\eta)<0$.

  In the previous examples we could avoid the ball constraint $r^2-\norm{\vb x - \xi}^2 \ge 0$, since in these cases $\cV_{\R}(h)$ is compact and the Archimedean condition is already satisfied.

\section{Computing the real radical}\label{sec:computing}

With the main ingredients, we can now describe the algorithm
for computing the real radical of an ideal $I=(\vb f)$, presented as the intersection of real prime ideals. The steps, summarised in Algorithm \ref{algo:real radical}, are detailed hereafter.

\begin{algorithm}[ht]\caption{\label{algo:real radical} Real radical}
\strong{Input:} Polynomials $\vb f=(f_{1}, \ldots,f_{s}) \subset \RRg$.\\

$d :=\max (\deg(\vb f_i), {i=1, \ldots,s})-1$; success := false; \\
Repeat until success
\begin{enumerate}
   \item $d:= d+1$
   \item Compute a generic element $\sigma^{*}$ of $\cL_{2d+2}(\pm \vb f)$
   \item Compute a graded basis $\vb k$ of $\ann_{d}(\sigma^{*})$ (\Cref{algo:orthogonal})
   \item Compute the numerical irreducible components $V_{i}$ of
     $V_{\CC}(\vb k)$ (described by witness sets)
   \item For each component $V_{i}$, check that $V_{i}$ is real (\Cref{algo:is real}). If not repeat from step (i).
   \item success := true
   \item For each component $V_{i}$ compute defining equations $\vb h_{i}=\{h_{i,1}, \ldots, h_{i,n+1}\}$ of $V_{i}$
\end{enumerate}
\strong{Output:} The polynomials $\vb h_i$ generating the minimal real prime ideals $\p_i$ lying over $(\vb f)$.
\end{algorithm}

In step (ii) we compute a generic element of $\cL_{2d+2}(\pm \vb f)$ solving a MOP with a constant objective function.

In step (iii) we use \Cref{algo:orthogonal} to compute the graded basis $\vb k$.

In step (iv) we find the irreducible components of the
variety $\cV_{\CC}(\vb k)$, described by witness sets (see e.g. \cite{BatesNumericallySolvingPolynomial2013}). The embedded components of $(\vb k)$ are not recovered by this technique.

In step (v) we control if the irreducible components of $\cV_{\C}(\vb k)$ are real, using \Cref{algo:is real}.

In step (vii), the equations defining $V_{i}$ are obtained from $n+1$
generic projections. In particular, the equation of a generic
projection of $V_{i}$ used in step (ii) of \Cref{algo:is real}
provides one of the defining equation, say $h_{i,1}$.

We prove the correctness of the algorithm. By \Cref{thm:realradical}
we have $\cV_{\R}(\vb k) = \cV_{\R}(\vb f)$ for $d\ge
\max(\deg(\vb f))$. Let $\p_i = (\vb h_i)$ in step (vii). By construction $\cV_{\R}(\vb k)=\bigcup_i (V_i)_{\R} = \bigcup_i \cV_{\R}(\p_i) = \cV_{\R}(\bigcap_i \p_i)$.
If step (v) succeeds, all the $\p_i$'s are real radical, and thus $\bigcap_i \p_i$ is real radical. Since $\cV_{\R}(\vb f) =\cV_{\R}(\bigcap_i \p_i)$, by the Real Nullstellensatz $\bigcap_i \p_i = \sqrt[\R]{\vb f}$ and the $\p_i$ are the real prime ideal lying over $(\vb f)$.
The loop stops for some $d\gg 0$ by \Cref{thm:realradical}.

\Cref{algo:real radical} computes the minimal real prime ideals lying over
$(\vb f)$, but does not check that the equations $\vb k$ define a real
radical ideal. If the ideal $(\vb k)$ has no embedded component and
the prime ideals $\p_{i}$ are of multiplicity 1 (checked with the
Jacobian criterion for $\vb h$ at a witness point of $\p_{i}$), then
the success of step (v) implies that $\vb k = \ann_d(\sigma^*)$
defines the real radical of $(\vb f)$.

\Cref{algo:real radical} can be simplified in the case where
$\cV_{\R}(\vb f)$ is finite. We can check that $(\vb k) =
\sqrt[\R]{\vb f}$, for $\vb k = \ann_d(\sigma^*)$, using the flat
extension criterion. We can also detect this condition with the
initial of $\vb k$, see \Cref{rem:flatext}. In this case, $\sigma^{*}$
extends to a positive linear functional on $\RRg$ and $(\vb k)= \sqrt[\R]{\vb f}$.

Similarly, when the ideal $(\vb k)$ is prime, one only needs to check
that it is real (using \Cref{algo:is real} on a generic
projection), steps (iv), (vii) can be skipped and we obtain $(\vb k) = \sqrt[\R]{\vb f}$.
When $(\vb k)$ is real radical, the algorithm can even output directly $(\vb k) = \sqrt[\R]{\vb f}$.


\section{Examples}\label{sec:example}

We illustrate Algorithm \ref{algo:real radical}, with the Julia
package \texttt{MomentTools.jl}\footnote{\url{https://gitlab.inria.fr/AlgebraicGeometricModeling/MomentTools.jl}},
using the Semi-Definite Programm optimizer \texttt{Mosek}.
\subsection{The isolated singular locus of a real surface}
\begin{example}
  Let $f = -10z^4 + x^3 - 3x^2z + 3xz^2 + 20yz^2 - z^3 - 10x^2 + 20xz - 10y^2 - 10z^2
, g = 5 - (x^2+y^2+z^2)$ and $S=\{\,\xi \in \R^3 \mid f(\xi)=0, \ g(\xi) \ge 0\,\}$. We want to compute the $S$-radical of $I=(f)$, which is equal to $(z-x, x^2-y)$.
  \begin{verbatim}
  X = @polyvar x y z
  f = -10*z^4 + x^3 - 3*x^2*z + 3*x*z^2 + 20*y*z^2
           - z^3 - 10*x^2 + 20*x*z - 10*y^2 - 10*z^2
  g = 5 - (x^2+y^2+z^2)
  v, M = minimize(one(f),[f], [g], X, 6, optimizer)
  sigma = get_series(M)[1]
  L = monomials(X,0:3)
  K,In,P,B = annihilator(sigma, L)
  \end{verbatim}
\vspace{-3mm}
We compute a generic positive linear functional $\sigma$ (by optimising the constant function $1$ on $S$), a graded basis \texttt{K} of $(\ann_d(\sigma))$,
 the initial monomials \texttt{In} of \texttt{K},
 a basis \texttt{P} of $\frac{\RRg}{(\ann_d(\sigma))} $ orthogonal with respect to $\langle \cdummy, \cdummy
 \rangle_{\sigma}$ and
 a monomial basis \texttt{B} of $\frac{\RRg}{(\ann_d(\sigma))}$.
The elements of \texttt{K} are:
{\small
  \begin{verbatim}
    z - 0.999999935776211x - 2.027089868945844e-9y
                              + 1.9280308682132505e-9
    x² - 1.9114608711668615e-8x - 0.9999998601127081y
                              - 2.6012502193917264e-7
\end{verbatim}
}%
\noindent{}These polynomials define a parametrisation of parabola
and thus generate a real radical ideal. They are approximation of the
generators of the $S$-radical of $I$ within an error \texttt{3.e-7}.

We can obtain the generators also using a slack variable $s$,
and replacying the inequality $g \ge 0$ by the equation $g - s^2 = 0$.
In this case the elements of \texttt{K} are:
{\small
  \begin{verbatim}
    z - 0.9999999987418964x - 2.0081938216111927e-9y
                                  + 1.848080975279204e-9
    x² + 5.417748642831503e-10x - 0.9999999813624691y
         - 4.507056024417168e-23s - 2.369265117430075e-8
    s²  + 2.531532655747432e-22ys - 7.729278487211091e-23xs
          - 2.0732509876020901e-22s + 0.9999999794170498y²
          + 1.1737503831818984e-8xy + 2.0000000080371674y
          - 1.4039307522382754e-8x - 4.999999978855321
\end{verbatim}
}%
\noindent{}and the generators of the $S$-radical are approximately $\texttt{K}\cap \R[x, y, z]$.
\end{example}

\begin{example}
  We compute equations for the hold of the Whitney umbrella. Let $f =
  x^2 - y^2z, g = 1 - (x^2+y^2+(z+2)^2)$ and $S=\{\,\xi \in \R^3 \mid
  f(\xi)=0, \ g(\xi \ge 0)\,\}$. We compute the $S$-radical of
  $I=(f)$, which is equal to $(x,y)$. Proceding as above, we obtain
  for \texttt{K}, the polynomials:
  {\small
  \begin{verbatim}
  x + 3.1388489268444904e-21,  y + 3.6567022687420305e-21
  \end{verbatim}
}%
\vspace{-4mm}\noindent{}These polynomials are a good approximation of the generators $(x,y)$ of
  the real radical, defining the singular locus of the Whitney umbrella.
\end{example}
\subsection{Components of different dimensions}
\begin{example}
  This example is taken from \cite[ex.~9.6]{rostalski_algebraic_2009}. We want to compute the real radical of $I = (f_1, f_2, f_3) \subset \R[x, y, z]$, where:
  \begin{align*}
  f_1 &= x^2 + x\,y - x\,z - x - y + z \\
  f_2 &= x\,y + 2\,y^2 - y\,z - x - 2\,y + z \\
  f_3 &= x\,z + y\,z - z^2 - x - y + z.
  \end{align*}
  Its variety has three irreducible components, two lines and a point, defined by the real prime ideals $\p_1=(x-z, y)$, $\p_2=(x-z+1, y-1)$ and $\m=(x-1, y-1, z-1)$. In the primary decomposition of $I$ there is an embedded component $\m'$, corresponding to the point $(1, 0, 1)\in \cV(\p_1)$ which has multiplicity two. The real radical of $I$ is $\sqrt[\R]{I} = \p_1 \cap \p_2 \cap \m = (y^2 - y, x^2 - 2xz + z^2 + x - z, xz + yz - z^2 - x - y + z, xy + xz - z^2 - 2x - y + 2z)$.

  We compute $\sqrt[\R]{I}$ as described in the algorithm.
{\small
  \begin{verbatim}
    v, M = minimize(one(f1),[f1,f2,f3], [], X, 8, optimizer)
    sigma = get_series(M)[1]
    L = monomials(X,0:3)
    K,I,P,B = annihilator(sigma, L)
\end{verbatim}
}%
\noindent{}The elements of \texttt{K} are:
        {\small
\begin{verbatim}
   xz - 0.9999999985579915x² - 0.9999999940764733xy
        + 0.9999999838152133x + 0.9999999868597321y
        - 0.9999999838041349z - 2.550976860304921e-10
   y² + 4.386341684978274e-7x² + 3.2135911001749273e-7xy
        - 8.511512801700947e-7x - 1.0000008530709377y
        + 9.888494964176088e-7z - 5.851033908621897e-8
   yz + 8.763853490689755e-7x² - 0.9999993625797754xy
        + 0.9999983122334805x - 1.6948939787209127e-6y
        - 0.9999980367703514z - 1.1680315895740145e-7
   z² - 0.9999991215344914x² - 1.99999935020258xy
        + 2.99999828318184x + 1.9999982828997438y
        - 2.999998007995895z - 1.1724998920381591e-7
\end{verbatim}
}
which are approximately (within an error of \texttt{1.e-6}) generators of $\sqrt[\R]{I}$.
\end{example}
\begin{example}
  This example is taken from \cite[8.2]{brake_validating_2016}. We want to compute the real radical of $I = (f_1, f_2, f_3) \subset \R[x, y, z]$, where:
  \begin{equation*}
  f_1 = xyz, \quad  f_2 = z(x^2+y^2+z^2+y), \quad f_3 = y(y+z).
  \end{equation*}
  The associated complex variety has four irreducible components: two
  conjugates lines intersecting in the origin, another line (double
  for $\vb f$) and a point. The real variety is given by a line $\p =
  (y, z)$ and a point $\m = (x, 2y+1, 2x-1)$. The real radical is
  $\sqrt[\R]{I}= \p \cap \m = (yx, z + y, y^2 + \frac{y}{2})$. A
  direct check shows that these polynomials generate a real radical ideal.

  We compute $\sqrt[\R]{I}$ as described above and obtain for \texttt{K}:
{\small
\begin{verbatim}
   z - 6.53338688785662e-19x + 0.9995827809845268y
                             - 0.00020850768649473272
   xy - 1.4685109255649737e-19x² + 5.9730164512226755e-6x
        + 2.1320912413237275e-19y + 1.0655056374451632e-19
   y² - 2.268705086623265e-6x² + 1.88498770272315e-19x
        + 0.4998194337295852y + 4.384653173789382e-6
\end{verbatim}
}
\noindent{}approximating (within an error of \texttt{5.e-4}) the generators of $\sqrt[\RR]{I}$.
\end{example}

\subsection{Limitations}
\Cref{algo:real radical} is a symbolic-numeric algorithm, which output
depends on the quality of the numerical tools that are involved. In
particular, the numerical quality of the generic positive linear functional
$\sigma^{*}$, produced by a SDP solver, impacts the computation of
generators of the real radical. This computation depends on a threshold used to determine
when a polynomial is in the annihilator. A detailled analysis of the
numerics behind the algorithm as well as an analysis of its complexity are left for futur investigations.
\begin{acks}
The authors would like to thank the anonymous referees
for their helpful suggestions. This work is
supported by the \grantsponsor{}{European Union's Horizon 2020 research and innovation programme}
{https://ec.europa.eu/programmes/horizon2020/en}
under the Marie Sk\l{}odowska-Curie Actions, grant agreement \grantnum{}{813211} (\grantnum{}{POEMA}).
\end{acks}
\bibliographystyle{ACM-Reference-Format}
\bibliography{POEMA.bib,Books.bib,algebraic.bib,numeric.bib,RealRadical.bib}

\end{document}